\documentclass[reqno]{amsart}
\usepackage[english]{babel} 
\usepackage[utf8]{inputenc} 
\usepackage[T1]{fontenc} 
\usepackage{lmodern} 
\usepackage{amssymb,amsfonts}
\usepackage{amsmath}
\usepackage{tikz}
\usepackage{graphicx}
\usepackage{psfrag}
\usepackage{cancel}
\usepackage{hyperref}
\usepackage{mathrsfs}
\usepackage{microtype}
\usepackage{mathtools}
\hypersetup{
    colorlinks=true,
    linkcolor=red,
    filecolor=magenta,      
    urlcolor=cyan,
}

\newtheorem{theorem}{Theorem}[section]
\newtheorem{definition}[theorem]{Definition}
\newtheorem{lemma}[theorem]{Lemma}

\newtheorem{proposition}[theorem]{Proposition}
\newtheorem{remark}[theorem]{Remark}

\numberwithin{equation}{section}

\theoremstyle{plain}
\newtheorem*{maintheorem}{Main Theorem}
\newcommand{\N}{\ensuremath{\mathbb{N}}}
\newcommand{\Z}{\ensuremath{\mathbb{Z}}}
\newcommand{\R}{\ensuremath{\mathbb{R}}}
\newcommand{\Q}{\ensuremath{\mathbb{Q}}}
\newcommand{\C}{\ensuremath{\mathbb{C}}}
\newcommand{\T}{\ensuremath{\mathbb{T}}}
\newcommand{\cH}{{\mathcal H}}
\newcommand{\cA}{{\mathcal A}}
\newcommand{\cC}{{\mathcal C}}
\newcommand{\cF}{{\mathcal F}}
\newcommand{\cP}{{\mathcal P}}
\newcommand{\cU}{{\mathcal U}}
\newcommand{\cV}{{\mathcal V}}

\newcommand{\cren}{{\mathcal R}_\text{cyl}}
\newcommand{\eps}{\epsilon}
\newcommand{\HH}{{\mathbb H}}
\newcommand{\CC}{{\mathbb C}}
\newcommand{\TT}{{\mathbb T}}
\newcommand{\ZZ}{{\mathbb Z}}
\newcommand{\RR}{{\mathbb R}}
\newcommand{\QQ}{{\mathbb Q}}
\newcommand{\NN}{{\mathbb N}}
\newcommand{\sm}{\setminus}
\newcommand{\ignore}[1]{}

\newcommand{\hol}{{\mathbf H}}

\newcommand{\dist}{\operatorname{dist}}
\newcommand{\diam}{\operatorname{diam}}
\newcommand{\cR}{{\mathcal R}}
\newcommand{\cL}{{\mathcal L}}
\newcommand{\cE}{{\mathcal E}}

\newcommand{\cQI}{{\mathcal{QI}}}
\newcommand{\cBT}{{\mathcal{BT}}}
\newcommand{\tl}{\tilde}

\newcommand{\bcu}{{\mathbf C}_{\mathcal U}}
\newcommand{\curr}{{\mathbf C}_{\mathcal U}^\RR}
\newcommand{\cu}{{\mathbf C}_{\mathcal U}^\RR}
\newcommand{\fxpt}{{\mathbf v}_*}
\newcommand{\bfv}{{\mathbf v}}

\title[Hyperbolicity of Renormalization for bi-cubic circle maps]{Hyperbolicity of Renormalization for bi-cubic circle maps with bounded combinatorics}

\author{Gabriela Estevez}
\address{Instituto de Matem\'atica e Estat\'istica, Universidade Federal Fluminense}
\curraddr{Rua Prof. Marcos Waldemar de Freitas Reis, S/N, 24.210-201, Bloco H, Campus do Gragoat\'a,
Niter\'oi, Rio de Janeiro RJ, Brasil}
\email{gestevez@id.uff.br}

\author{Michael Yampolsky}
\address{Department of Mathematics, University of Toronto}
\curraddr{40 St George Street, Toronto, Ontario, Canada}
\email{yampol@math.toronto.edu}
 

\thanks{G.E. was partially financed by the Coordenação de Aperfeiçoamento de Pessoal de Nível Superior - Brasil (CAPES) - Finance Code 001. M.Y. was partially supported by NSERC Discovery Grant.}


\begin{document}

\begin{abstract} We construct a hyperbolic attractor of renormalization of bi-cubic circle maps with bounded combinatorics, with a codimension-two stable foliation.
\end{abstract}

\maketitle
\section{Introduction} Renormalization theory, which has revolutionized one-dimensional dynamics,
has developed around two key examples: unimodal maps of the interval, and maps of the circle with a single non-degenerate critical point (critical circle maps).
Renormalization techniques for critical circle maps first appeared in the works \cite{ORSS} and \cite{FKS}, and were further developed by O.E.~Lanford  \cite{Lan86}, who described a conjectural renormalization picture for such maps in full generality. This set of conjectures, known as Lanford's Program, became the subject of intensive study for several decades (see the historical account in \cite{Yam02}). A cornerstone of Lanford's Program is the existence of a horseshoe attractor for renormalization of critical circle maps with irrational rotation numbers. The symbolic dynamics on the attractor is described by (a compactification of) the natural extension of the action of the Gauss map $G(\rho)=\{1/\rho\}$ on the rotation numbers of the maps. In a suitable Banach manifold, the attractor is hyperbolic with a single unstable direction, which can be informally thought of as the direction of change of the rotation number of the map. The maps with the same irrational rotation number form stable manifolds of the horseshoe, and renormalization hyperbolicity promotes the topological equivalence of such maps to a smooth one with an appropriate degree of smoothness.

Lanford's Program was completed by the second author in \cite{Yam02,Yam03}. A natural direction of generalizing it is to consider maps with more than one critical point -- in which case, adding each new critical point would create an additional combinatorial invariant for the maps, which should result in an extra unstable direction for renormalization.

The problem of extending Lanford's Program is already very interesting, and conceptually difficult, for maps with two critical points.
In \cite{ESY}, D.~Smania and the two authors proved that, restricted to the set of bi-cubic maps whose 
rotation numbers are irrationals of a type bounded by some $B\in\NN$, renormalization has an attractor $\cA_B$.
However, the attractor $\cA_B$ contains maps with a single critical point of order $9$. These appear as limits of renormalizations of bi-cubic maps, for which the two critical orbits ``collide''. As was shown by the second author \cite{Yam03}, in the space of unicritical analytic circle maps, renormalization has only one unstable direction at these points. Near these maps, as critical points move (combinatorially) closer, the two unstable cone fields collapse into one. 
 The same problem would, of course, arise for any number of critical points greater than one.

In this paper, we overcome the difficulty by working in a different function space, in which the collapse does not happen, extending the definition of renormalization operator to this space, and then proving that the attractor {\sl is} uniformly hyperbolic for the new definition. As has been the second author's experience in the subject, the
correct choice of a function space and an appropriate definition of renormalization play a crucial role.\footnote{This was the case with the definition of cylinder renormalization in \cite{Yam02} to prove the hyperbolicity result for unicritical circle maps; or in the paper \cite{GY} in which renormalization had to be defined in {\it three} different function spaces.} The ideas which we develop in this paper is applicable to renormalization horseshoes for maps with arbitrary criticality, 
constructed by I.~Gorbovickis and the second author in \cite{GY21}.

Our main result can be summarized as follows:
\begin{maintheorem}\label{maintheorem}
  Let $B\in\N$. There exists a suitable function space into which the space of bi-cubic maps considered in \cite{ESY} naturally embeds, in which:
  \begin{itemize}
  \item renormalization is a compact analytic operator;
  \item for each $B\in\NN$, the attractor for renormalization of bi-cubic maps of type bounded by $B$ is uniformly hyperbolic with  stable foliation of codimension two.
    \end{itemize}
\end{maintheorem}


\subsection{Organization of the paper}
In Section \ref{sec:preliminaries}, we give basic definitions which are used in the rest of the paper. In the following two sections, \ref{sec:realcomplexpairs} and \ref{section: cyl ren}, we review the theory of renormalization for commuting pairs (real and complex) and cylinder renormalization, respectively. In section \ref{sec:renortriples}, we  define renormalization in a suitable space of triples. Finally, in section \ref{sec-hyperb} we prove our Main Theorem in the space of triples. 

\subsection*{Aknowledgement.} The authors would like to thank I.~Gorbovickis for generously sharing his ideas and helping to fix an argument in an earlier version of the paper.

    Part of this article was developed while G.E. had a postdoctoral position at Instituto de Matem\'atica, Universidade Federal do Rio de Janeiro.

\section{Preliminaries}\label{sec:preliminaries}
\subsection{Bi-cubic circle maps}
When we talk about the circle, we will usually mean the affine manifold
$\T=\R/\Z$, which is canonically identified with the unit circle $S^1$ via the exponential map $\theta\mapsto \exp(2\pi i\theta)$.
We let $R_\rho(x)=x+\rho\mod \Z$ denote the rotation by an angle $\rho\in\T$.
Given a circle homeomorphism $f$,
we denote by $\rho(f)$ its rotation number.
A  bi-cubic  circle map is an orientation preserving circle homeomorphism of class $C^\omega$  with two critical points both of which are of cubic type.
To fix the ideas, we will place one of the two critical points at $0\in\T$. Note that a  $C^\omega$ map of $\T$ is a germ of a complex-analytic (and real-symmetric) map of
an annular neighborhood of $\T$ in $\C/\Z$ to $\C/\Z$.

Suppose, $f$ is a bi-cubic map with an irrational rotation number $\rho(f)\notin\Q$ whose critical points are $0$ and $c$. By a theorem of  Yoccoz \cite{Yoc1984}, there exists an orientation preserving homeomorphism $\phi:\T\to\T$ which is a
conjugacy between $f$ and a rigid rotation:
$$\phi\circ f\circ \phi^{-1}=R_{\rho(f)}.$$
We will say that the {\it signature} of $f$ is the pair $\cC(f)=(\rho(f),\delta)$, where $\delta\in(0,1)$ is the length of the arc of $\T$ going from $0$ to $\phi(c)$ in
positive (counterclockwise) direction. Clearly, $f$ is topologically conjugate to another bi-cubic map $g$ by an orientation preserving map fixing $0$, and sending the other critical point of $f$ to a critical point of $g$ if and only if $\cC(f)=\cC(g)$.


Examples of bi-cubic maps are given by a family of 
Blaschke products
\begin{equation}\label{Zakeri}
B(z)= e^{2 \pi i t} z^{3} \left(\frac{z-p}{1-\overline{p}z}\right)
\left(\frac{z-q}{1-\overline{q}z} \right),
\end{equation}
where $p,q \in \mathbb{C}$ with $|p|>1$, $|q|>1$ and $t \in [0,1]$. It was considered by Zakeri in \cite{Zak99}, who proved the following statements:

\begin{lemma}\label{Zakericubic}
  Given $\rho \in [0,1) \setminus \Q$ and $0<\delta<1$, there exists a unique map $B$ of type \eqref{Zakeri} such that the
signature $\cC(B)=(\rho,\delta)$.
\end{lemma}

\begin{lemma}\label{Zakerimodel}
If a Blaschke product $G$ is topologically conjugate to an element of the family \eqref{Zakeri}, then $G=B \circ R$, where $B$ is of type \eqref{Zakeri} and $R(z) =e^{ir}z$, for $r \in \mathbb{R}$.
\end{lemma}
 

\subsection{Continued fraction expansions}
For a finite or infinite continued fraction we will use the abbreviation
\begin{equation*}
           \cfrac{1}{a_{0}+\cfrac{1}{a_{1}+\cfrac{1}{ \ddots} }} \equiv \left[a_0,a_1,\ldots  \right]
\end{equation*}
 and in what follows we will always assume $a_j\in \N$.
Note that an irrational 
number in $(0,1)$ has a unique infinite continued fraction expansion; a continued fraction expansion of a rational number is
necessarily non-unique. This is due to the fact that $[n]=[n-1,1]$; we will always use the shorter expansion for a rational.
The Gauss map $$G(x)\equiv \left\{\frac{1}{x}\right\}$$
acts as a forgetful map on the continued fraction expansions of numbers in $(0,1)$:
$$G([a_0,a_1,a_2,\ldots])=[a_1,a_2,\ldots].$$

Suppose that  the continued fraction expansion of a number 
$\rho=[a_0,a_1,\ldots]$ has at least $m$ terms; its {\it $k$-th continued fraction convergent} for $k\leq m$ is the rational number
$$\frac{p_k}{q_k}\equiv [a_0,a_1, \cdots ,a_{k-1}].$$
Recall, that the denominators 
$\{q_k\}_{k \in \N}$ satisfiy the following recursive formula:
\begin{equation}
  \label{qk-recursion}
 q_0=1, \hspace{0.3cm} q_1=a_0 \hspace{0.3cm} \text{and} \hspace{0.3cm} q_{k+2}=a_{k+1}q_{k+1} +q_{k} \hspace{0.3cm} \text{for all $k\geq 1$.}
\end{equation}

If $\rho=\rho(f)\in(0,1)$ is the rotation number of a circle homeomorphism $f$, then $f^{q_k}(x)$ is a {\it closest return} of a point $x\in \T$ in the following sense.
Suppose $q_{k+1}$ is defined (that is, the continued fraction expansion of $\rho$ has at least $k+1$ terms). Denote $I_k(x)$ the arc of $\T$  whose endpoints are $x$, $f^{q_k}(x)$
and such that $f^{q_{k+1}}(x)\notin I_k(x)$. Then for all $j<q_k$, the iterate $f^j(x)\notin I_k(x)$.

\begin{figure}[h]
  \includegraphics[width=0.8\textwidth]{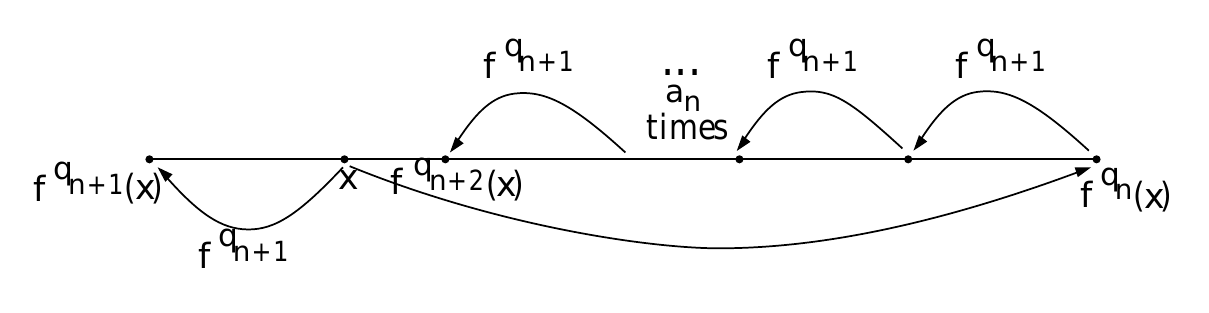}
\caption{\label{fig-returns} Closest returns of a point $x\in\T$}
  \end{figure}

If we denote $t_\ell$ the iterates $q_k+\ell q_{k+1}$, for $0\leq \ell\leq a_{k+1}$ (compare with (\ref{qk-recursion}) then $f^{t_\ell}(x)\in I_k(x)$ and all of these iterates are closest returns in the above sense: there is no smaller iterate of $x$ in the subinterval of $I_k(x)$ which is bounded by $x$ and $f^{t_\ell}(x)$. In fact, all of the closest returns of a point $x\in\T$  are of this form, see Figure~\ref{fig-returns}.

It is not difficult to see, that
the collection of intervals
\[
 \mathcal{P}_k(x) \ = \ \left\{ f^{i}(I_k(x)):\;0\leq i\leq q_{k+1}-1 \right\} \;\bigcup\; 
 \left\{ f^{j}(I_{k+1}(x)):\;0\leq j\leq q_{k}-1 \right\} 
\]
is a partition of the circle by closed intervals intersecting only at their endpoints. It is called the {\it dynamical partition of level (or depth) $k$\/} corresponding to the point $x$ (see \cite[Section 1.1, Lemma 1.3, page 26]{dMvS} or \cite[Appendix]{EdFG}).

For notational convenience, we will always choose $x=0$ in the above, and will write $I_k\equiv I_k(0)$ and $\mathcal{P}_k\equiv\mathcal{P}_k(0)$. The latter will  be referred to as the dynamical partition of level (or depth) $k$.

Let us denote $\cQI$ the set of quadratic irrationals in $(0,1)$, that is, the set of numbers whose continued fraction expansion is eventually periodic. In other words, $\rho\in\cQI$ iff there exist $k\geq 0$ and $m\in\N$ such that $G^{k}(\rho)=G^{k+m}(\rho)$.
For $B\in\N$ we will say that $\rho\in(0,1)\setminus \Q$ is of a {\it type bounded by } $B$ if the coefficients in its continued fraction expansion are
bounded by $B$. We denote the set of such numbers by $\cBT_B$. We also set
$$\cBT=\cup_B\cBT_B$$
and call the numbers in $\cBT$ irrationals of {\it bounded type}. A number $\rho$ is of bounded type iff it satisfies the Diophantine condition
with exponent $2$.

\section{Real and complex commuting pairs}
\label{sec:realcomplexpairs}

\subsection{Renormalization of bi-critical commuting pairs}\label{sec:renorcp} 
Renormalization of multicritical circle maps is defined in the language of {\it commuting pairs} which originated 
in \cite{ORSS}. Here we briefly recall the main definitions.

For $a\neq b\in\R$ let us denote $[a,b]$, $(a,b)$ the closed and open intervals  with boundary points $a$, $b$ regardless of the orientation.
\begin{definition}\label{multcommpairs} A multicritical commuting pair  $\zeta=(\eta,\xi)$ consists of two  orientation preserving interval homeomorphisms $\xi:I_\xi \rightarrow  \xi(I_\xi)$ and $\eta : I_\eta \rightarrow \eta(I_{\eta})$ of class $C^\omega$
satisfying:
 \begin{enumerate}
  \item $(\eta(0),\xi(0))\ni 0$, and $I_{\xi}=[\eta(0), 0]$, $I_{\eta}=[0, \xi(0)]$;
  \item the origin is a cubic critical point both for $\eta$ and for $\xi$;
  \item $\xi$ and $\eta$ satisfy the commutation property  $\eta \circ \xi=\xi \circ \eta$ where defined;
    \item $\xi(\eta(0))=\eta(\xi(0))\in I_\eta$.
 \end{enumerate}
\end{definition}
A multicritical circle map can be associated to a multicritical commuting pair $\zeta=(\eta|I_\eta,\xi|I_\xi)$ as follows. Let $x\in I_\xi$ be
any regular point of $\xi$, and glue the endpoints of the interval $I=[x,\xi(x)]$ via $\xi$. More precisely, consider the one-dimensional analytic manifold $S$, obtained by using $\xi$ as a local chart in a small neighborhood of $x$.
Consider the map $f\colon I\to I$ defined as follows:
$$
f(x) = \begin{cases}
\eta(x) &\text{if } x\in I_\eta \text{ and }\eta(x)\in I \\
\xi(\eta(x)) &\text{if } x\in I_\eta \text{ and }\eta(x)\not\in I \\
\eta(\xi(x)) &\text{if } x\in I_\xi.
\end{cases}
$$
Let $\tilde f\colon S\to S$ be the projection of $f$ to a map of the analytic circle $S$, and let $h$ be a diffeomorphism of the same smoothness as $\zeta$, which maps $S\to \T$ and sends $0\in S$ to $0\in\T$. It is straightforward to check that 
the map $\tilde f$ projects to a well-defined multicritical circle map:
\begin{equation}
  \label{eq:fzeta}
f_\zeta=h\circ \tilde f\circ h^{-1}\colon\T\to\T
\end{equation}
via  $h$.
Of course, $f_\zeta$ is not defined uniquely, since there is {\it a priori} no canonical way of identifying the manifold $S$ with the circle $\T$. Rather, we recover a conjugacy class of multicritical cicle maps from a pair $\zeta$ from this gluing construction.

We say that a pair $\zeta$ is {\it bi-cubic} if $f_\zeta$ is a bi-cubic map (see Figure~\ref{fig:bcp}).
The rotation number $\rho(\zeta)$ of a commuting pair is defined as
$$\rho(\zeta)=\rho(f_\zeta),$$
note that the latter is well-defined, since $f_\zeta$ is defined up to a conjugacy. If $\rho(\zeta)\notin \Q$, the signature $\cC(\zeta)$ is defined as
$$\cC(\zeta)=\cC(f_\zeta);$$
it is similarly well-defined.

Let $f$ be a  multicritical circle map with a rotation number $\rho=\rho(f)$ whose continued fraction expansion has at least $m+1$ terms, and as usual, denote $p_n/q_n$ the continued fraction convergents of $\rho$, defined for $n\leq m+1$.
Let $\widehat{f}$ be the lift of $f$ such that $0< \widehat{f}(0)<1$ (note that $D\widehat{f}(0)=0$). For $1\leq n\leq m$, let $\widehat{I}_{n}$ be the closed interval in $\R$, containing the origin as one of its boundary points, which is projected onto $I_{n}(0)\subset \T=\R/\Z$.
We define $\xi: \widehat{I}_{n+1}\rightarrow \R$ and $\eta: \widehat{I}_{n} \rightarrow \R$ by $$\xi= T^{-p_{n}}\circ \widehat{f}^{q_{n}}\text{ and }\eta= T^{-p_{n+1}}\circ \widehat{f}^{q_{n+1}},$$ where $T(x)=x+1$ is the unit translation. Then the pair
\begin{equation}
  \label{eq:zetafn}
  \zeta_{f,n}=(\eta|_{\widehat{I}_{n}(c_j)}, \xi|_{\widehat{I}_{n+1}(c_j)})
\end{equation}
is a multicritical commuting pair. In what follows, for the simplicity of the notation, we will  denote this pair simply by $(f^{q_{n+1}}|_{I_n}, f^{q_n}|_{I_{n+1}})$.

Suppose now that  $f$ is bi-cubic, then the pair $\zeta_{f,n}$ corresponds to a multicritical circle map $f_{\zeta_{f,n}}$ (\ref{eq:fzeta}) which has either two cubic critical points, or a single critical point of order $9$. The latter happens if the orbit of one of the critical points of $f$ hits the other one for an iterate $m\leq q_{n+1}$.

\begin{figure}[!ht]
\includegraphics[width=0.6\textwidth]{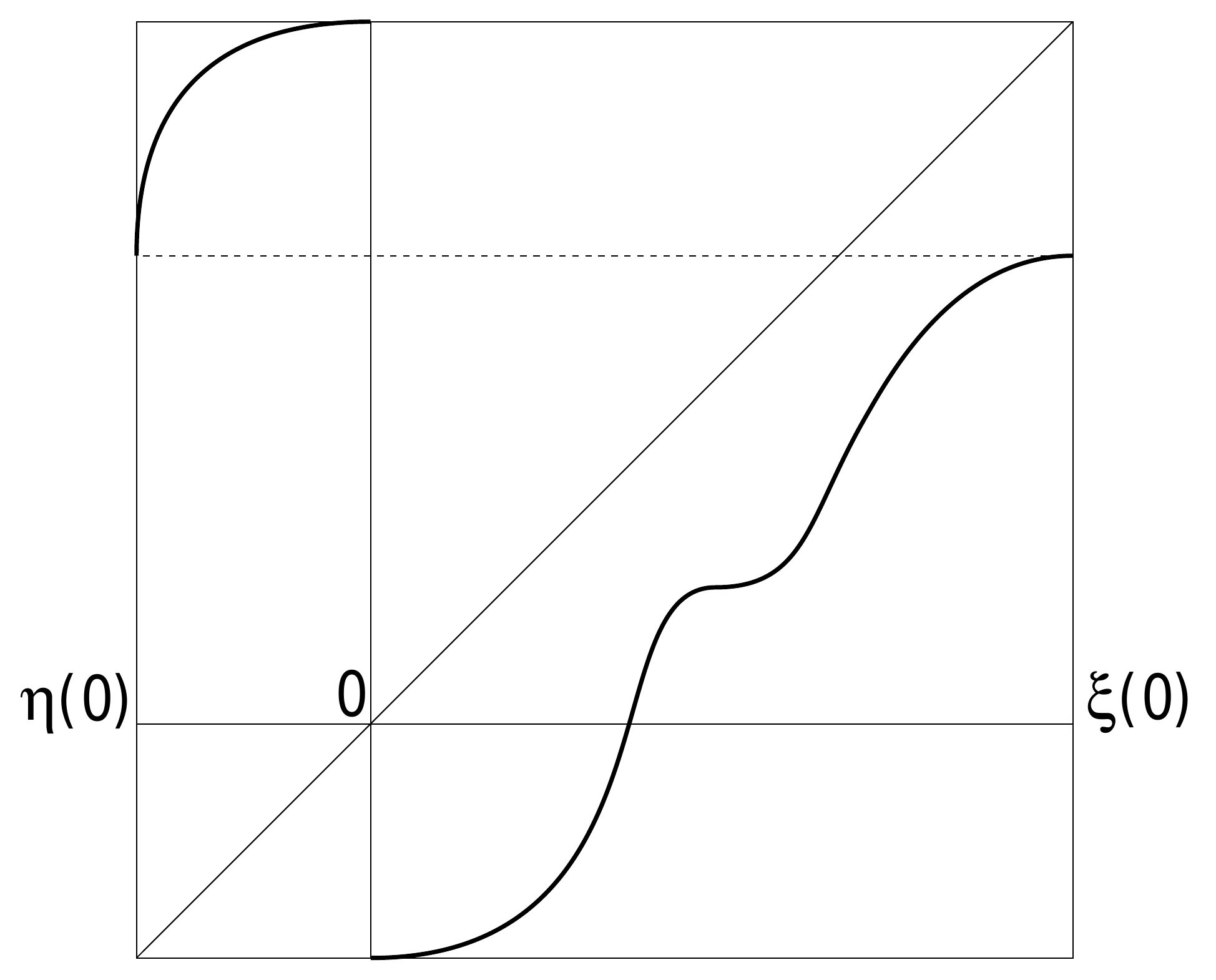}
  \caption{\label{fig:bcp} A bi-cubic commuting pair.}
\end{figure}

For any given commuting pair $\zeta=(\eta,\xi)$ we denote by $\widetilde{\zeta}$ the pair $(\widetilde{\eta}|_{\widetilde{I_{\eta}}}, \widetilde{\xi}|_{\widetilde{I_{\xi}}})$, where tilde means linear rescaling by the factor $1/\xi(0)$. 
Thus, $\widetilde{\xi}(0)=1$ and $\widetilde{\eta}(0)=-|I_{\xi}|/|I_{\eta}|$. We call $\widetilde{\zeta}$ a {\it normalized} commuting pair.
\begin{definition} The \emph{height} $\chi(\zeta)$ of a commuting pair $\zeta=(\eta, \xi)$ is equal to the natural number $a$ such that
  $$ (\eta^{a+1}(\xi(0)),\eta^{a}(\xi(0)))\ni 0,$$
  if it exists.
  When no such number exists, we set $\chi(\zeta)=\infty$.
\end{definition}
Note that, when it is finite,  the height of the commuting pair
$(f^{q_{n+1}}|_{I_n}, f^{q_n}|_{I_{n+1}})$ is equal to the coefficient $a_{n+1}$ in the continued fraction expansion of $\rho(f)$, and 
$\chi(\zeta)=\infty$ if and only if $\eta$ has a fixed point in its domain of definition.

\begin{definition}\label{defren} 
Let $\zeta=(\eta, \xi)$ be a commuting pair with $\chi(\zeta)=a< \infty$. We define the \emph{pre-renormalization} of $\zeta$ as the pair$$p\mathcal{R}(\zeta) =(\eta|_{[0, \eta^{a}(\xi(0)) ]} \ , \ \eta^{a}\circ \xi|_{I_{\xi}} ).$$
Moreover, we define the \emph{renormalization} of $\zeta$ as the normalization of $p \mathcal{R}(\zeta)$:$$\mathcal{R}(\zeta)= \left(\widetilde{\eta}|_{[0,\widetilde{\eta^{a}(\xi(0))} ]} \ , \ \widetilde{\eta^{a}\circ \xi}|_{\widetilde{I}_{\xi}}
  \right).$$
\end{definition}

If $\zeta$ is a bi-cubic commuting pair with $\chi(\mathcal{R}^{j}\zeta) < \infty$ for $0 \leq j \leq n-1$, we say that $\zeta$ is \textit{$n$-times renormalizable}. Moreover, if $\chi(\mathcal{R}^{j}\zeta)< \infty$  for all $j \in \N$, we say that $\zeta$ is \textit{infinitely renormalizable}.
It is straightforward to check that
the rotation number $\rho(\zeta)$ is equal to a finite or infinite continued fraction
$$\rho(\zeta)=[\chi(\zeta), \chi(\mathcal{R}\zeta), \cdots, \chi(\mathcal{R}^{n}\zeta), \cdots ],$$
and that renormalization acts as the Gauss map $G(\rho)$ on the rotation numbers:
$$\rho(\cR\zeta)=G(\rho(\zeta)).$$

The $n$-th renormalization $\cR^nf$ of  a bi-cubic circle map $f$
is the commuting pair
$$\mathcal{R}^{n}f=\left(\widetilde{f}^{q_{n+1}}|_{\widetilde{I_n}}, \widetilde{f}^{q_{n}}|_{\widetilde{I_{n+1}}}\right).$$
This definition is clearly consistent with the above definition of renormalization of commuting pairs:
$$\cR(\cR^nf)=\cR^{n+1}f,$$
and
$$\rho(\cR f)=G^2(\rho(f)).$$
The action of renormalization on signatures is given by the expanding cocycle
\begin{equation}
  \label{eq:cocycle}
  (\rho,\delta)\mapsto \left( G(\rho),\left\|\frac{\delta}{\rho} \right\|_{\RR/\ZZ}\right).
\end{equation}
(compare with \cite[Section 3]{dFG} where the corresponding cocycle is constructed for renormalization acting on bi-critical commuting pairs).

Compactness properties of infinite orbits of a renormalization transformation are known as {\it a priori} bounds; such bounds allow us to discuss attractors of renormalization, and play a crucial role in establishing rigidity results. Real {\it a priori} bounds 
for an infinitely renormalizable bi-cubic map concern pre-compactness of renormalizations in the $C^r$ topology.
Real {\it a priori} bounds were proved for multicritical circle maps  with an arbitrary irrational rotation number in \cite{EdF}. These bounds were improved in \cite{EdFG} to be universal (independent of the map). We give a precise statement for bi-cubic circle maps:
\begin{theorem}\label{th:realbds}
 There exists $n_0 \in \NN$ such that for any $n\geq n_0$ and any bi-cubic circle map $f$ with an irrational rotation number, the sequence $\{f^{q_{n+1}}|_{I_n} \}$ is universally bounded in the $C^1$ norm.
\end{theorem}

Complex {\it a priori} bounds concern convergence of analytic extensions of renormalizations in the complex plane, and will require a detailed discussion.

\subsection{Complex {\it a priori} bounds and convergence of renormalization} 
Let us recall the definition of a holomorphic commuting pair given in \cite{Yam2019}, which generalizes the original definition of de~Faria \cite{dF}. 

\begin{definition}\label{def:hcp}
 Given an analytic multicritical commuting pair $\zeta=(\eta|_{I_{\eta}}, \xi_{I_{\xi}})$, we say that it extends to a \emph{holomorphic commuting pair} $\mathcal{H}$, if there exist three simply-connected and $\R-$symmetric domains $D,U,V\subseteq \C$, whose intersections with the real line are denoted by $I_U=U \cap \R$, $I_V=V \cap \R$ and $I_D=D \cap \R$ and a simply connected $\R-$symmetric Jordan domain $\Delta$ that satisfy the following,
\begin{enumerate}
 \item the endpoints of $I_U$ and $I_V$ are critical points of $\eta$ and $\xi$, respectively;
 \item  $\overline{D}, \overline{U}, \overline{V}$ are contained in $\Delta$; $\overline{U}\cap \overline{V} = \{0\} \subseteq D$; the sets $U \setminus D, V\setminus D, D \setminus U$ and $D \setminus V$ are non-empty, connected and simply-connected; $I_{\eta} \subset I_U \cup \{0\}$, $I_{\xi}\subset I_V \cup \{0\}$;
 \item $U \cap \mathbb{H}$, $V \cap \mathbb{H}$ and  $D \cap \mathbb{H}$ are Jordan domains;
 \item the maps $\eta$ and $\xi$ have analytic extensions to $U$ and $V$, respectively, so that $\eta$ is a branched covering map of $U$ onto $(\Delta \setminus \R) \cup \eta(I_U)$, and $\xi$ is a branched covering map of $V$ onto $(\Delta \setminus \R)\cup \xi(I_V)$,
   with all the critical values of both maps contained in the real line;
 \item the maps $\eta:U \to \Delta$ and $\xi: V \to \Delta$ can be analytically extended to branch covering maps
   $$\hat\eta:\hat U\to(\Delta \setminus \R) \cup \eta(I_U)\text{ and }\hat \xi:\hat V\to (\Delta \setminus \R)\cup \xi(I_V),$$
   with only real critical values, and $\Delta\Supset \hat U\ni 0$, $\Delta\Supset \hat V\ni 0$;
 \item the domains $\hat U\cap\hat V\supset D$ and 
   the map $\nu=\widehat{\eta}\circ \widehat{\xi}= \widehat{\xi}\circ \widehat{\eta}$ is defined in $D$ and is a branched covering of $D$ onto $(\Delta\setminus \R)\cup \nu(I_D)$ with only real critical values.
\end{enumerate}
\end{definition}

\begin{figure}[!ht]
\includegraphics[width=0.7\textwidth]{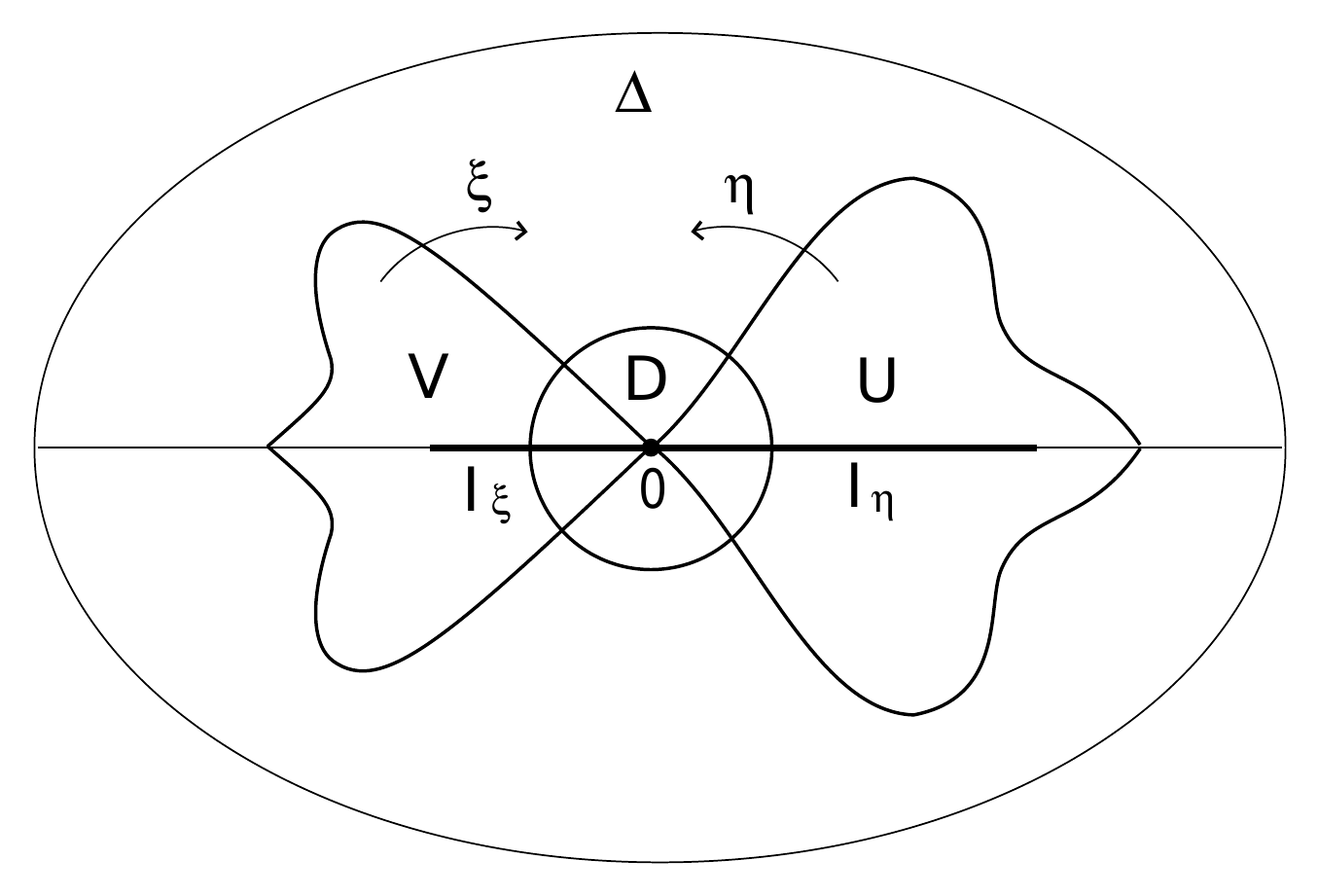}
  \caption{\label{fig:hcp} A holomorphic commuting pair.}
\end{figure}

\noindent
We shall  call $\zeta$ the {\it commuting pair underlying $\cH$}, and write $\zeta\equiv \zeta_\cH$. When no confusion is possible, we will use the same letters $\eta$ and $\xi$ to denote both the maps of the commuting pair $\zeta_\cH$ and their analytic extensions to the corresponding domains $U$ and $V$.



We can naturally view a holomorphic pair $\cH$ as three triples
$$(U,\xi(0),\eta),\;(V,\eta(0),\xi),\;(D,0,\nu).$$
We say that a sequence of holomorphic pairs converges in the sense of Carath{\'e}odory convergence, if the corresponding triples do.
We denote the space of triples equipped with this notion of convergence by $\hol$.

We let the \emph{ modulus} of a holomorphic commuting pair $\cH$, which we denote by $\mod(\cH)$ to be the modulus of the largest annulus $A\subset \Delta$,
which separates $\CC\setminus\Delta$ from $\overline\Omega$.

\begin{definition}\label{H_mu_def}
For $\mu\in(0,1)$ let $\hol(\mu)\subset\hol$ denote the space of holomorphic commuting pairs
${\cH}:\Omega_{{\cH}}\to \Delta_{\cH}$, with the following properties:
\begin{enumerate}
\item $\mod (\cH)\ge\mu$;
\item {$|I_\eta|=1$, $|I_\xi|\ge\mu$} and $|\eta^{-1}(0)|\ge\mu$; 
\item $\dist(\eta(0),\partial V_\cH)/\diam V_\cH\ge\mu$ and $\dist(\xi(0),\partial U_\cH)/\diam U_\cH\ge\mu$;
\item {the domains $\Delta_\cH$, $U_\cH\cap\HH$, $V_\cH\cap\HH$ and $D_\cH\cap\HH$ are $(1/\mu)$-quasidisks.}
\item$\diam(\Delta_{\cH})\le 1/\mu$;
\end{enumerate}
\end{definition}

Let the {\it degree} of a holomorphic pair $\cH$ denote the maximal topological degree of the covering maps constituting the pair. Denote by $\hol^K(\mu)$ the subset of $\hol(\mu)$ consisting of pairs whose degree is bounded by $K$. The following is an easy generalization of Lemma 2.17 of \cite{Yam03}:  

\begin{lemma}
\label{bounds compactness}
For each $K\geq 3$ and $\mu\in(0,1)$ the space $\hol^K(\mu)$ is sequentially compact.
\end{lemma}

\noindent
We say that a real commuting pair $\zeta=(\eta,\xi)$ with
an irrational rotation number has
{\it complex {\rm a priori} bounds}, if there exists $\mu>0$ such that all renormalizations of $\zeta=(\eta,\xi)$ extend to 
holomorphic commuting pairs in $\hol(\mu)$.
The existense of complex {\it a priori} bounds is a key analytic issue 
 of renormalization theory. 
 For holomorphic commuting pairs corresponding to bi-cubic circle maps, with bounded type rotation number, complex a priori bounds were established \cite{ESY}:
 \begin{theorem}
   \label{th:bounds}
Let $\zeta$ be a bi-cubic analytic commuting pair with $\rho(\zeta)\in\cBT_B$. Then $\zeta$ has complex {\it a priori} bounds, with $\mu=\mu(B)$.
   \end{theorem}
 Based on this result, the following renormalization convergence statement was obtained in \cite{ESY}:
 \begin{theorem}[\bf Exponential convergence of the renormalizations]\label{theorem:convergencerenor}
Let $\zeta_1$ and $\zeta_2$ be two analytic bi-cubic circle maps with same bounded type rotation number $\rho\in\cBT_B$. Then there exist $n_0>0$, $C>0$ and $0<\lambda=\lambda(B)<1$ such that for all $n\geq n_0$ we have
\begin{equation}\label{eqconvrenor}
 \dist_C(\mathcal{R}^n\zeta_1, \mathcal{R}^n\zeta_2 ) \leq C \, \lambda^n,
\end{equation}
where $\dist_C$ stands for the Carath{\'e}odory distance.
\end{theorem}
 Theorems \ref{th:bounds} and \ref{theorem:convergencerenor} imply the following:
 \begin{theorem}
   \label{th:attractor}
   Let $B\in\N$. There exists a Carath\'eodory compact set $\cA_B$ of analytic commuting pairs such that the following holds:
   \begin{itemize}
   \item each $\zeta\in\cA_B$ lies in $\hol(\mu)$ with $\mu=\mu(B)$ as above;
   \item $\cA_B$ is $\cR$-invariant: $\cR(\cA_B)=\cA_B$;
     \item for every bi-cubic pair $\zeta_1$ with $\rho(\zeta_1)\in\cBT_B$, the renormalizations $\cR^n(\zeta_1)$ converge to $\cA_B$ geometrically fast in the Carath\'eodory sense, with rate $\lambda(B)$.
     \end{itemize}
   \end{theorem}
 Note, that pairs in $\cA_B$ can have two cubic critical points, or only one, of order $9$ -- the maps latter correspond to the situations when the two cubic critical points ``collide'' in the limit.

 \begin{remark}
   \label{rem:symmetry}
   Note also that every pair $\zeta=(\eta,\xi)\in\cA_B$ consists of maps which decompose as
   $$\eta=\phi_\eta(z^3),\;\xi=\phi_\xi(z^3),$$
   where $\phi_\eta$ and $\phi_\xi$ are holomorphic in a neighborhood of the origin. The space of pairs with this cubic symmetry property is clearly renormalization-invariant, and limit points of renormalization belong to this space by the standard machinery of renormalization theory. 
   \end{remark}

 Let us denote $\cP_B\subset \cA_B$ the set of periodic orbits of renormalization with rotation numbers in $\cBT_B$.
As shown in \cite{Yam2019}, all periodic signatures are represented in $\cP_B$.
 From the standard considerations of compactness, as well as the density of periodic orbits of the cocylcle (\ref{eq:cocycle}), we have
 \begin{proposition}\label{rem:per-dense}
The set $\cP_B$ is dense in $\cA_B$ in the Carath{\'e}odory sense.
   \end{proposition}
 
\section{Cylinder renormalization of bi-cubic maps}
\label{section: cyl ren}
In this section, we recall the definition of cylinder renormalization that acts in the space of bi-cubic analytic circle maps. It was introduced in \cite{Yam2019}, and is based on the previous definition for the unicritical case in \cite{Yam02} (see also \cite{GY} for a more general exposition).

\subsection{Some functional spaces}
Firstly, let us denote $\pi$ the natural projection $\CC\to\CC/\ZZ$. For an equatorial annulus $U\subset \CC/\ZZ$
let ${\mathbf A}_U$ be the space of bounded analytic maps $\phi:U\to \CC/\ZZ$ continuous up to the boundary, such that 
$\phi(\TT)$ is homotopic to $\TT$, equipped with the uniform metric.
We shall turn ${\mathbf A}_U$ into a real-symmetric complex Banach manifold as follows. 
Denote $\tl U$ the lift $\pi^{-1}(U)\subset \CC$. The space of functions $\tl\phi:\tl U\to\CC$ which are analytic,
continuous up to the boundary, and $1$-periodic, $\tl\phi(z+1)=\tl\phi(z)$, becomes a Banach space when endowed
with the sup norm. Denote that space $\tl{\mathbf A}_U$. For a function $\phi:U\to\CC/\ZZ$ denote
$\check\phi$ an arbitrarily chosen lift $\pi(\check\phi(\pi^{-1}(z)))=\phi$.
Observe that $\phi\in{\mathbf A}_U$ if and only if $\tl\phi=\check\phi-\operatorname{Id}\in\tl{\mathbf A}_U$.
We use the local homeomorphism between $\tl{\mathbf A}_U$ and ${\mathbf A}_U$ given by
$$\tl\phi\mapsto \pi\circ(\tl\phi+\operatorname{Id})\circ\pi^{-1}$$ 
to define the atlas on  ${\mathbf A}_U$. The coordinate change transformations
are given by $\tl\phi(z)\mapsto \tl\phi(z+n)+m$ for $n,m\in\ZZ$, therefore with this atlas 
${\mathbf A}_U$ is a real-symmetric complex Banach manifold; we denote ${\mathbf A}_U^\RR$ its real slice, consisting of the real-symmetric maps
$\phi(\TT)=\TT$.

\noindent
\begin{definition}
  \label{cylren}
Given a  map  $f \in {\mathbf A}_U$ we say that 
it is {\it cylinder renormalizable} with period $k$ if the following holds:
\begin{itemize}
\item there exists $m\in \NN$ such that the rotation number $\rho(f)$ has at least $m+1$ digits in its continued fraction expansion, and
  $k=q_m$;
\item there exist repelling periodic points $p_1$, $p_2$ of $f$ in $U$ with periods $k$
and a simple arc $l$ connecting them such that $f^k(l)$ is a simple arc, and  $f^k(l)\cap l=\{p_1,p_2\}$;
\item the iterate $f^k$  is defined and univalent in a neighborhood $N$ of the curve $l$; if we denote $C_f$ the domain  bounded by $l$ and $f^k(l)$,
then the quotient of $(C_f\cup N\cup f^k(N))\sm\{p_1,p_2\}$ by $z\sim f^k(z)$ is a 
Riemann surface conformally isomorphic to the cylinder $\CC/\ZZ$ 
(we will  call a domain $C_f$ with these properties
a {\it fundamental crescent of $f^k$});
\item for a point $z\in \bar C_f$ with $\{f^j(z)\}_{j\in\NN}\cap \bar C_f\ne \emptyset$,
set $R_{C_f}(z)=f^{n(z)}(z)$ where $n(z)\in\NN$ is the smallest value for which $f^{n(z)}(z)\in\bar C_f$.
We further require that 
there exists a point $c$ in the domain of $R_{C_f}$ such that $f^m(c)=0$ for some $m<n(c)$.
\end{itemize}

Denote $\hat f$ the projection of $R_{C_f}$ to $\CC/\ZZ$ with $c\mapsto 0$. 
We will say that $\hat f$ is a {\it cylinder renormalization} of $f$
with period $k$.

Note that if, in the above, $f\in{\mathbf A}_U^\RR$ is a multicritical circle map, then $\hat f$ is also a multicritical circle map.
\end{definition}

\noindent
The relation of the cylinder renormalization procedure to renormalization of pairs is as follows:
\begin{proposition}[\cite{Yam02}]
  \label{prop cylren1}
Suppose $f$ is a multicritical circle map with rotation number $\rho(f)\in \RR\setminus\QQ$.
Assume that it is cylinder renormalizable with period $q_n$. Then the corresponding renormalization
$\hat f$ is also a multicritical circle map with rotation number $G^n (\rho(f))$. Also, 
$\cR\hat f$ (in the sense of pairs) is analytically conjugate to $\cR^{n+1}f$.
\end{proposition}

\noindent
The next property of  cylinder renormalization is the following:
\begin{proposition}[\cite{Yam02}]
\label{ren open set}
Let $f$ be cylinder renormalizable with period $k=q_n$, with a cylinder renormalization $\hat f$ corresponding to the fundamental 
crescent $C_f$, and let $W$ be any equatorial annulus compactly contained in the domain $U_1$ of $\hat f$.
Then there is an open neighborhood $G$ of $f$ in ${\mathbf A}_U$ such that every map $g\in G$ is cylinder
renormalizable, with a fundamental crescent $C_g\subset U$ which depends continuously on $g$ in the
Hausdorff sense. Moreover, there exists a holomorphic motion 
$\chi_{g}:\partial C_f\mapsto \partial C_g$ over $G$, such that 
$\chi_g(f(z))=g(\chi_g(z))$.
And finally, the renormalization $\hat g$ is contained in ${\mathbf A}_W$.
\end{proposition}


\noindent
The next proposition shows that $g\mapsto \hat g$ is well-defined:
\begin{proposition}
\label{well-defined}
In the notations of Proposition~\ref{ren open set}, let $C'_g$ be a different family of 
fundamental crescents, which also depends continuously on $g\in G$ in the Hausdorff
sense. Then there exists an open neighborhood $G'\subset G$ of $f$, such that 
the cylinder renormalization of $g\in G'$ corresponding to $C'_g$ is also $\hat g$.
\end{proposition}

\noindent
A key property of the cylinder renormalization is the
following (see Proposition~6.7 of \cite{Yam2019}):

\begin{proposition}
  \label{analytic dependence2}
  In the above notation,
 the correspondence $g\mapsto\hat g$ is a locally analytic operator from a neighborhood of $f\in\mathbf A_U$ to
    $\mathbf A_W$.
   
\end{proposition}


\begin{proposition}
  \label{conj-maps}
  Let $f$ and $g$ be two bi-critical circle maps which are analytically conjugate
  $$\phi\circ f=g\circ \phi$$
  in a neighborhood $U$ of the circle. Suppose that $f$ is cylinder renormalizable with period $k=q_n$, and let $C_f$ be a corresponding fundamental crescent of $f$. Suppose $C_f\cup f(C_f)\Subset U$. Then, $g$ is also cylinder renormalizable with the same period and with a fundamental crescent  $C_g=\phi(C_f)$; and denoting $\hat f$ and $\hat g$ the corresponding cylinder renormalizations of $f$ and $g$ respectively, we have
  $$\hat f\equiv\hat g.$$
\end{proposition}
\begin{proof}
  The domain $C_g$ is a fundamental crescent by definition. Denote $C_f'$, $C_g'$ the closures of the two fundamental crescents minus the endpoints.
  The analytic conjugacy $\phi$ projects to an analytic isomorphism
  $$\hat\phi:C'_f\cup f^k(C'_f)/f^k\simeq\CC/\ZZ\longrightarrow C'_g\cup g^k(C'_g)/g^k\simeq\CC/\ZZ,$$
  which conjugates $\hat f$ to $\hat g$. By Liouville's Theorem, $\hat\phi$ is a translation; since it fixes the origin, it is the identity.
\end{proof}

\subsection{Cylinder renormalization operator}

The cylinder renormalization of a bi-cubic commuting pair is defined in the same way as that of an analytic bi-cubic map.:
\begin{definition}
Let $\zeta=(\eta,\xi)$ be a bi-cubic holomorphic commuting pair which is at least $n$ times renormalizable for some $n\in\NN$.
Denote $D_\eta$ the domain of the pair $\zeta$.
Denote $$\zeta_n=(\eta_n,\xi_n)\equiv p\cR^n(\zeta):D_{\eta_n}\to\CC.$$
We say that $\zeta$ is cylinder renormalizable with period $k=q_n$ if:
 \begin{enumerate}
\item there exist two compex conjugate fixed points $p^+_n$, $p^-_n$ of the map $\eta_n$ with $p^\pm_n\in\pm\HH$,
and an $\RR$-symmetric simple arc $l$ connecting these points such that $l\subset {D_{\eta_n}}$ and 
$\eta_n(l)\cap l=\{p^+_n,p^-_n\}$;
\item denoting $C$ the $\RR$-symmetric domain bounded by $l$ and $\eta(l)$ we have $C\subset D_{\eta_n}$,
and the quotient of $\overline{C\cup\eta(C)}\sm\{p^+_n,p^-_n\}$ 
by the action of $\zeta_n$ is a Riemann surface homeomorphic
to the cylinder $\CC/\ZZ$ (that is, $C$ is a fundamental crescent of $\eta$).
\end{enumerate}
Clearly, the action of $\zeta$ induces a bi-critical circle map $f$ with rotation number $\rho(\cR^n\zeta)$ on the quotient $$(\overline{C\cup\eta(C)}\sm\{p^+_n,p^-_n\})/\zeta_n\simeq \CC/\ZZ;$$
we call $f$ the cylinder renormalization of $\zeta$ with period $q_n$.

\end{definition}
Note that the existence of the two conjugated points required in  (1) is guaranteed for all holomorphic commuting pairs without fixed points in the real line (see e.g. \cite[Section 4.3]{GY}).

\begin{theorem}
  \label{th:hcp-cyl}
  Let $\zeta$ be any bi-cubic  commuting pair which extends to a holomorphic commuting pair $\cH_\zeta$. Then $\zeta$ is cylinder renormalizable with period $q_0=1$. Furthermore, for every $\mu>0$ there exists an equatorial annulus $U$ such that if $\cH_\zeta\in\hol(\mu)$, then the cylinder renormalization lies in $\mathbf A_U^\RR$.
\end{theorem}
The proof follows {\it mutatis mutandis} the proof of Lemma 4.10 of \cite{GY}, so we will not reproduce it here.

\noindent
Let $\cA_B$ be the attractor of renormalization for pairs with rotation numbers in $\cBT_B$ (Theorem~\ref{th:attractor}). Compactness considerations imply:

\begin{proposition}
  \label{renorm-operator}
    There exists an equatorial annulus $U$, $m\in\NN$, and a neighborhood $\cV=\cV(\cA_B)$ in Carath\'eodory topology such that the following holds:
  \begin{enumerate}
  \item every $\zeta\in \cV$ is cylinder renormalizable with period $1$     and the cylinder renormalization $f_\zeta\in {\mathbf A}_{U};$
  \item further, the analytic map $f_\zeta$ is cylinder renormalizable with period $l(f)=p_m/q_m$ (where $p_m/q_m$ is the $m$-th convergent of $\rho(f)$), which is locally constant in $\cV$; its cylinder renormalization $\hat f_\zeta$ is defined in $U_1\Supset U$.
    \end{enumerate}
\end{proposition}
Let us fix $B\in\NN$. By complex {\it a priori} bounds, and in view of the above discussion,
we have the following:
\begin{proposition}
  \label{prop:cren}
There exist $m=m(B)\in\NN$, $\mu=\mu(B)>0$, an $\RR$-symmetric  disk $\Delta\subset \CC$, and an equatorial $\RR/\ZZ$ symmetric annulus $U$ such that for every multicritical circle map $f\in{\mathbf A}_U$ with either two cubic or one degree nine critical points such that 
   there are at least $m+1$ terms in the continued fraction expansion of $\rho(f)$ all bounded by $B$
  the following holds:
\begin{enumerate}
\item   $\cR^{m-1}f$ extends to a holomorphic pair $\cH_f\in\hol(\mu)$ with range $\Delta$;
\item $\cH_f$ is cylinder renormalizable with period $1$, and the corresponding cylinder renormalization $g_f\in\mathbf A_{\tl U}$ with $\tl U\Supset U$.
\end{enumerate}
\end{proposition}
Let us denote \begin{equation}
  \label{eq:crenop}
  g_f=\cren f,
\end{equation}
and call it the {\it cylinder renormalization} of $f$. For every $\zeta \in\mathcal A_B$, denote $f_\zeta$ the cylinder renormalization of $\zeta$ with period $1$ in $\mathbf A_U$, and let $$\tl{\mathcal A}_B=\{f_\zeta\text{ for }\zeta\in\mathcal A_B\}.$$ Clearly, the cylinder renormalization operator $\cren$ is an analytic operator on an open neighborhood of $\tl{\mathcal A}_B$ in $\mathbf A_U$.

As noted in the introduction, the attractor of renormalization $\tilde\cA_B$ is not going to be uniformly hyperbolic under the action of $\cren$. Indeed,
$\tilde\cA_B$ contains unicritical circle maps with a critical point of order $9$.  For these orbits of $\cren$ the two combinatorial conjugacy  invariants  collapse into one. In the next section we will define a different Banach manifold, into which the dynamics of $\cren$ will conveniently embed -- and in which the main result will hold.

\section{The space of triples}
\label{sec:renortriples}
Let us introduce several more Banach spaces. For an equatorial annulus $A\subset \CC/\ZZ$, we will let $\mathbf D_A$ be the space of all bounded holomorphic functions from $A\to\CC$. If $\tl A\subset\CC$ is the lift of $A$, then $\mathbf D_A$ can be identified with the space of $1$-periodic bounded holomorphic functions in $\tl A$. We will equip the latter with the sup norm, which turns $\mathbf D_A$ into a Banach space.

Similarly, for a domain $U\subset \CC$, we let $\mathbf G_U$ denote the space of bounded holomorphic functions $U\to\CC/\ZZ$, again with the natural Banach structure. Finally, $\mathbf B_U$ will denote the standard Banach space of bounded holomorphic functions from a domain $U\subset\CC$ to $\CC$ with the sup norm.

\noindent
Let $\mathbf q(x)=x^3$.

\noindent
\begin{definition}
  Let $\cU=(U_1, U_2,U_3)$ be an ordered triple of an $\RR/\ZZ$-symmetric equatorial annulus $U_1$ and two $\RR$-symmetric domains $U_2$, $U_3$ in $\CC$.
  We denote
$\mathbf C_\cU$ the space of triples $(\phi_1,\phi_2,\phi_3)$ with $\phi_1 \in \mathbf D_{U_1}$, $\phi_2\in{\mathbf B}_{U_2}$, $\phi_3\in\mathbf G_{U_3}$ such that:
  \begin{itemize}
  \item $\mathbf q\circ \phi_1(U_1)\Subset U_2$ and $\mathbf q\circ \phi_2\circ \mathbf q\circ \phi_1(U_1)\Subset U_3$;
    \item $\phi_1(0)=0$;
      \item each of the maps $\phi_i$ is locally conformal in a neighborhood of every real point of its domain of definition.
    \end{itemize}

\end{definition}
Clearly, $\mathbf C_\cU$ is an open subset of ${\mathbf D}_{U_1}\times {\mathbf B}_{U_2}\times {\mathbf G}_{U_3}$, and thus possesses a natural product Banach manifold structure. It is real-symmetric, and we denote $\curr$ its real slice, consisting of triples of real-symmetric maps. Denote  the composition
\begin{equation}
  \label{cyl-map}
g_{(\phi_i)}\equiv \phi_3\circ\mathbf q\circ \phi_2\circ \mathbf q\circ \phi_1:U_1\to \CC/\ZZ.
  \end{equation}
We note:
\begin{proposition}
  \label{projection1}
  The correspondence $\Gamma:\mathbf C_\cU\to {\mathbf A}_{U_1}$, given by
  $$(\phi_i)\mapsto g_{(\phi_i)}$$
  is analytic.
  
\end{proposition}

We will  interpret cylinder renormalization as an operator acting on triples in the real slice $\mathbf C_\cU^\RR$. 

Now, let $f$ be in the domain of $\cren$ (see Proposition~\ref{prop:cren}) and set $(\eta,\xi)=\cR^nf$. The map $\eta$ is a branched covering of a domain $W_\eta\to\Delta$ of degree $9$. Hence, it uniquely decomposes as
$$\eta=\psi_3\circ \mathbf q\circ \psi_2\circ \mathbf q\circ\psi_1,$$
where $\psi_i$ are conformal maps onto the images; each partial composition is a branched covering; the domain of $\psi_1$ is $W_\eta$, the range of $\psi_3$ is $\Delta$. Now let $$H:C_f\to \CC/\ZZ$$ be the uniformizing coordinate, and let
$(\phi^f_1,\phi^f_2,\phi^f_3)$ be the triple $$(\psi_1\circ H^{-1},\psi_2,H\circ \psi_3)\in\mathbf C_\cU$$ where $\cU=(U_1,U_2,U_3)$ with $U_1=U$, and the domains $U_2$, $U_3$ appropriately and uniformly chosen. We set
$$\Lambda (f)=(\phi^f_1,\phi^f_2,\phi^f_3),$$
and refer to it as a {\it lift} of $f$ to the space of triples.
Note that from definition, the lift of  cylinder maps to the space of triples is injective.

With the definition of the lift at hand, we can define renormalization in a appropriate subspace of the triples space $\mathbf C_\cU^\RR$. By definition of $\cren$ (see equation \eqref{eq:crenop}), Lemma \ref{prop:cren} and Proposition \ref{projection1}, we have that there exists a neighborhood $\mathcal{V}_{B}$ of $\hat{\mathcal A}_B\equiv \Lambda(\widetilde{{\mathcal A}}_B)$ such that,  the operator given by
\begin{equation}
\begin{split}
\hat\cR \colon \mathcal{V}_B &\to  \mathbf C_\cU \\
\mathbf v & \mapsto  \Lambda(\cren(\Gamma(\mathbf v))),
\end{split}
 \label{eq:hatcr}
\end{equation}
is well defined and it is analytic. 
We will prove our Main Theorem for the renormalization operator $\hat\cR$ defined in \eqref{eq:hatcr}, and acting in the  Banach manifold $\mathcal{V}_B$. Note that, similarly to cylinder renormalization, $\hat\cR$ naturally extends to a complex-analytic and real-symmetric transformation.


\subsection{Maps with periodic rotation numbers}

We will now specialize to real-symmetric maps, and consider the real slice $\curr$ of the space $\bcu$ consisting of maps with real symmetry, with the induced real Banach manifold structure (cf. \cite{Yam02,GY}). As seen from Proposition~\ref{rem:per-dense}, periodic orbits of $\hat \cR$ are dense in $\hat{\mathcal A}_B$; let us denote the set of these orbits by $\hat{\mathcal P}_B$. Let $\fxpt\in\hat{\mathcal P}_B$ be a periodic point of $\hat\cR$ with period $p\in\NN$. Recall that we denote by $l=q_m$ the period of $\hat\cR$.
Let us set $$f_*=\Gamma(\fxpt),$$
and denote $\rho_*=\rho(f_*)$ and let $\cC$ be the signature of $f_*$.
We 
 define $$D_*=\{ \mathbf v\in\curr,\text{ such that }\rho(\Gamma(\mathbf v))=\rho_*\},$$
and
$$S_*=\{\mathbf v\in\curr,\text{ such that }\cC(\Gamma(\mathbf v))=\cC\}\subset D_*.$$
As shown in \cite{ESY}, all the maps in $S_*$ are $C^{1+\alpha}$-conjugate with each other, with $\alpha=\alpha(B)$.

\noindent
We will see that the local stable set of $\fxpt$ is a submanifold of $\curr$:
\begin{theorem}
\label{stable manifold}
There is an open neighborhood $W\subset \curr$ of $\fxpt$ such that $S_*\cap W$
is an analytic submanifold of $\curr$ of codimension $2$.
It is the local stable manifold of $\hat \cR$ at $\fxpt$: for every $\bfv\in S_*\cap W$ the iterates
$$\hat\cR^{pk}\bfv\underset{k\to\infty}{\longrightarrow}\fxpt\text{ at a geometric rate.}$$
Moreover, for $\mathbf v\in S_*$ denote $f=\Gamma(\mathbf v)$, and let $\phi_f$ be the unique
smooth conjugacy between $f$ and $f_*$ which fixes $0$. Then the dependence $\mathbf v\mapsto \phi_f$ is analytic.
\end{theorem}

The proof of this theorem will be completed in the following section. We begin working on it by describing the manifold structure of $D_*$.

\noindent
Let us define a cone field $\cE$ in the tangent bundle of $\bcu$ given by the condition:
$$\cE=\{\nu\in T\bcu|\; \inf_{x\in\RR}D\Gamma \nu (x)>0\}.$$
We first show that:
\begin{theorem}\label{stable1}
  The set $D_*$ is a local submanifold of $\curr$ of codimension $1$. For every $\mathbf v\in D_*$, the intersection
  $T_{{\mathbf v}}D_*\cap \cE=\emptyset$.
  \end{theorem}

As before denote $p_k/q_k$ the reduced form of the $k$-th
continued fraction convergent of $\rho_*$. Furthermore, 
define $D_k$ as the set of ${\mathbf v}\in\cu$ for which $\rho(g)=p_k/q_k$ and 
$0$ is a periodic point of $g$ with period $q_k$, where $g=\Gamma({\mathbf v}).$
As follows from the Implicit Function Theorem, this is a codimension $1$ submanifold of $\cu$.
\begin{lemma}
\label{cone not in tk}
Let ${\mathbf v}\in D_k$, and denote $T_{{\mathbf v}}D_k$ the tangent space to $D_k$ at this point.
Then $T_{{\mathbf v}}D_k\cap \cE=\emptyset$.
\end{lemma}
\begin{proof}
Let $\nu\in \cE$ and suppose $\{ {\mathbf v}_t\}\subset \cu$ is a one-parameter family
such that  $${\mathbf v}_t={\mathbf v}+t\nu+o(t),$$
and set $f=\Gamma({\mathbf v})$ and $f_t=\Gamma({\mathbf v}_t)$, so that
$$\pi^{-1}\circ f_t\circ\pi=\pi^{-1}\circ f\circ\pi+tD\Gamma(\nu)+o(t).$$
Then for sufficiently small values of $t$, 
$\pi^{-1}\circ f_t\circ \pi>\pi^{-1}\circ f\circ\pi$. Hence $f_t^{q_k}(0)\neq f^{q_k}(0)$
and thus $f_t\notin D_k$.
\end{proof}


\noindent
Elementary considerations of the Intermediate Value Theorem imply that 
for every $k$ there exists a value of $\theta\in(0,1)$ such that 
the map $f_\theta=R_{\theta}\circ \Gamma(\fxpt)\in D_k$. 
Moreover,
if we denote $\theta_k$ the angle  with the smallest absolute value satisfying this 
property, then $\theta_k\to 0$.
Set $${\mathbf v}_k=(\phi_1,\phi_2,R_{\theta_k}\circ \phi_3).$$
For $k$ large enough, ${\mathbf v}_k\in\cu\cap D_k$.
Let $T_k=T_{{\mathbf v}_k}D_k$.
Fix $v\in\cE$. By  Lemma~\ref{cone not in tk} and the Hahn-Banach Theorem there exists $\eps>0$
such that for every $k$ there exists a linear functional $h_k\in (T\cu)^*$
with $||h_k||=1$,
such that $\operatorname{Ker}h_k=T_k$ and $h_k(v)>\eps$.
By the Alaoglu Theorem, we may select a subsequence
$h_{n_k}$ weakly converging to $h\in (T\cu)^*$. Necessarily 
$v\notin \operatorname{Ker}h$, so
$h\not\equiv 0$. Set $T=\operatorname{Ker}h$.

\smallskip
\noindent
{\it Proof of Theorem~\ref{stable1}.} 
By the above, we may select a splitting $T\cu=T\oplus v\cdot\RR$.
Denote $p:T\cu\to T$ the corresponding projection, and let
$\psi:\cu\to T\cu$ be a local chart at $\fxpt$.
Lemma~\ref{cone not in tk} together with the Mean Value Theorem imply that
$p\circ \psi:D_k\to T$ is an isomorphism onto the image,
and there exists an open neighborhood $\cU$ of the origin in $T$, such that 
$p\circ \psi(D_k)\supset \cU$. Since the graphs $D_{k}$ are analytic, we may
select a $C^1$-converging subsequence $D_{k_j}$ whose limit is a smooth graph $G$ over $\cU$.
 Necessarily, for every $g\in G$, $\rho(g)=\rho_*$. As we have seen above,
every point $g\in D_*$ in a sufficiently small neighborhood of $\fxpt$ is in
$G$, and thus $G$ is an open neighborhood in $D_*$.
\begin{flushright} $\Box$\end{flushright}

\section{Hyperbolicity of $\hat\cR$.}
\label{sec-hyperb}
Recall that $\hat\cP_B\subset\curr$ denotes the set of 
periodic points of $\hat\cR$ with rotation numbers of type bounded by $B$. We will show the following:
\begin{theorem}\label{th:unifhyp}
Let $B\in\NN$. The set $\hat\cP_B$ is uniformly hyperbolic with two-dimensional unstable bundle.
\end{theorem}
Informally speaking, one unstable direction corresponds to adding a constant to $\phi_3$, and the other one corresponds to adding a constant to $\phi_2$.
The key in the above theorem, and the difference from the result of \cite{Yam2019} is the uniformity of hyperbolicity.  The two unstable cones constructed in \cite{Yam2019} collide when the two critical points combine into one in the limit.

Let us $\fxpt$ be as above and 
denote
$$\cL\equiv D\hat\cR|_{\bfv_*}.$$
We have (see \cite{GY}):
\begin{proposition}
  \label{prop-comp-lin}
The linear operator $\cL$ is compact.

\end{proposition}
\begin{proof}
By  \cite[Proposition 9.1]{Yam02}, see also \cite[Proposition 6.4]{GY}, the linear operator $D\cren|_{f_*}$ is compact. On the other hand, the operator $D \Lambda|_{\cren(f_*)}$ is bounded. Therefore, the operator $\cL$ being  the composition of a bounded operator with a compact operator, is compact.
  \end{proof}

Spectral theory of compact operators implies, in particular, that $\cL$ is a strict contraction in an invariant subspace of a finite codimension $\kappa$ corresponding to the part of the spectrum of $\cL$ inside the unit disk. Furthermore, there are at most $\kappa$  non-contracting eigendirections, whose eignespaces span a $\kappa$-dimensional non-contracting invariant subspace.

As was shown in \cite{GY} (see also \cite{Yam03}):
\begin{proposition}
\label{dir1}
  There exist constants $n_1\in\NN$, $a_1>0$ and $\lambda_1>1$ such that for every $\nu\in \cE\cap T_{\bfv}\curr$  and $n>n_1$
  we have
  $$\inf\cL^n\nu>a_1\lambda_1^n\inf_{x\in\RR} D\Gamma\nu(x).$$
\end{proposition}
Hence, the cone field $\cE$ contains an unstable direction of $\hat\cR$.

 We now proceed to describing the manifold structure of $S_*\subset D_*$. Let $c_f\neq 0$ denote the non-zero critical point of $f$. 
Let $S_{n}\subset D_*$  consist of maps $f$ such that
$f$ has the same combinatorics as $\cE$ up to the level $n$, and the $n$-th renormalization $\cR^n(f)$ is uni-critical (that is, two critical orbits collide).

As a straightforward consequence of Implicit Function Theorem considerations, we note:
\begin{proposition}\label{propconj2}
  The set $S_{n}$ is a local submanifold of $D_*$ of codimension $1$.
  \end{proposition}

Since the $n$-th renormalization of $f\in S_{n}$ is a uni-critical circle mapping, we can apply the results of \cite{Yam02} to conclude:
\begin{theorem}
  \label{conj4a}
  All maps $f\in\Gamma(S_{n})$ are smoothly conjugate. Moreover, let $\hat f$ be an arbitrary base point in $D_{n}$. Then, for $f\in S_{n}$  the conjugacy $\phi_f$ which realizes
  $$\phi_f\circ f\circ\phi_f^{-1}=\hat f$$
 and sends $0$ to $0$ can be chosen so that the dependence $f\mapsto \phi_f$ is analytic.
  \end{theorem}
Furthermore,
\begin{theorem}
  \label{conj4}
  The local submanifold $S_{n}$ lies in a strong stable manifold of $\hat \cR$. The differential $D\hat \cR$ is contracting in its tangent bundle, and the contraction is uniform (depending only on the bound on the rotation number).

  \end{theorem}
The contraction was shown in \cite{Yam02}. Its uniformity is a consequence of uniformity of convergence shown in \cite{ESY}.

We will now demonstrate that the submanifolds $S_n$ locally converge to a codimension one submanifold of $D_*$ at   $\mathbf v_*\in S_*$. In view of the uniformity of contraction in Theorem~\ref{conj4}, it is enough to demonstrate that:
\begin{theorem}
  \label{expdir1} For any $k\in\NN$, there
    exists $\nu\in T_{\mathbf v_*}D_*$ such that
  $$||\cL^k\nu||>2||\nu||.$$
  \end{theorem}
\begin{proof}
  The proof is inspired by the following simple observation. As before, let $l=q_m$ be the period of the cylinder renormalization of $\cren$. Suppose that the second critical point $c$ of $f_*$ lies in the interval $I_j=[0,f_*^{q_j}(0)]$ for $j>m$. Denote $\mathbf v_*=(\phi_1,\phi_2,\phi_3)$ and set
  $$(\phi_1^1,\phi_2^1,\phi_3^1)=\hat\cR(\mathbf v_*).$$
  Then, the diffeomorphism $\phi_2^1$ is obtained by an almost linear rescaling of a restriction of $\phi_2$ to a proper subset of its domain of definition. Hence, a small change in $\phi_2$ will be amplified by renormalization, leading to an expansion in this coordinate.

  In the general case, let us denote
  $$(\phi_1^l,\phi_2^l,\phi_3^l)=\hat\cR^l(\mathbf v_*).$$
  The map $\phi_2^l$ is a rescaling of a composition of restrictions of a sequence of diffeomorphisms $\phi_j$ to intervals of the dynamical partition of level $mk$. Let $\phi_j|_I$ be the final diffeomorphism in this composition ($I\in\cP_{mk}$). Consider a vector field $\nu=(\nu_1,\nu_2,\nu_3)\in T_{\mathbf v_*}D_*$ in which $\nu_j$ is a ``bump'' localized to a small neighborhood of $I$, and a change in $\nu_3$ to keep the rotation number constant. Considerations of bounded distortion both in the dynamical and the parameter space, imply that the norm of $\cL^k\nu$ is bounded below by $a^k$ with $a>1$, and the claim follows.
\end{proof}

As a consequence, the limit of $S_n$ at $\mathbf v_*$ is a codimension one local submanifold $W$ of $D_*$. By Theorem~\ref{conj4}, $W$ lies in the strong stable manifold of $\hat \cR$. We thus have:
\begin{theorem}
  \label{expdir2}
  The linearization $\cL$ has at most two non-stable eigendirections. One of them is repelling, the other one (in $TD_*$) may be either repelling or neutral.
\end{theorem}
\begin{proof}
  Since $\cL$ contracts in $W$, all eigendirections in $TD_*$ are contracting except at most one. By way of contradiction, assume that the remaining one is also contracting. Note that the local strong stable manifold $W^{\text{ss}}_\text{loc}(\mathbf v_*)\subset S_*$, so this would imply that in a neighborhood of $\mathbf v_*$, the manifold $D_*$ lies in $S_*$ -- which is clearly impossible.

\end{proof}
To complete the proof of  Theorem~\ref{stable manifold}, we show:
\begin{theorem}
  \label{expdir3}
The operator $\cL$ does not have any neutral eignedirections.
  \end{theorem}
\begin{proof}
The argument given by M.~Lyubich in \cite{Lyu99} in the context of unimodal maps applies {\it mutatis mutandis}. We recap its steps. Assume the contrary. Consider the complexification of the operator $\hat\cR$ which we will denote by the same notation. By Lyubich's Slow Small Orbits Theorem, for every open neighborhood $U=U(\mathbf v_*)$, there exists an orbit $(\mathbf v_n)_{n\in\NN}$ of $\hat \cR$ which is contained entirely in $U\setminus W^{\text{ss}}(\mathbf v_*)$. If $U$ is chosen sufficiently small, the elements of this orbit possess complex {\it a posteriori} bounds. This implies that for each $n$, the complex cylinder map $f_n=\Gamma(\mathbf v_n)$ has an invariant quasicircle. Passing to a limit, we obtain a bi-infinite tower with complex bounds, corresponding to the limit point of $f_n$. By Tower Rigidity Theorem, such a tower is unique, and its base map is $\mathbf v_*$. This means that $\mathbf v_n\in W^{\text{ss}}(\mathbf v_*)$, which is the desired contradiction.

  \end{proof}

Since periodic orbits are dense in $\hat{\mathcal A}_B$, we can use the standard considerations of uniform hyperbolicity to derive the main
result. We can use, for instance,  Lemma VII.1 of \cite{Man83} which discusses uniform hyperbolicity in Banach space setting:

\begin{lemma}\cite[Lemma VII.1, page 564]{Man83}\label{ManeLemma}
  Let $(E,d)$ be a Banach space, $U\subseteq E$ an open set, $f: U \to E$ a $C^{k+\gamma}$ map and $K \subseteq U$ satisfying the following properties:
  \begin{itemize}
   \item [a)] $f^{n}(K)\subseteq U$, for all $n \geq 0$,
   \item [b)] For all $n\geq0$ and $x \in E$ there exists a splitting $E= H^{(n)}_{x} \oplus G_{x}^{(n)}$ such that:
   \begin{itemize}
    \item [$\rm b_1$)] $H^{(0)}_{x}$ and $G_{x}^{(0)}$ depend continuously on $x$.
    \item [$\rm b_2$)] For all $n\geq 0$ and $x \in f^{n}(K)$
    \[
      (D_xf)H_x^{(n)}=H_{f(x)}^{(n+1)} \ \text{and}  \
      (D_xf)G_x^{(n)} \subset G_{f(x)}^{(n+l)}.
    \]
    \item [$\rm b_3$)] $(D_xf)|_{H_x^{(n)}}$  is injective for all $n\geq 0$ and $x \in f^{n}(K)$.  
    \item [$\rm b_4$)] There exist $C>0$, $\mu>0$ and $0< \varepsilon < \gamma \mu/3$ satisfying
    \[
      \|D_{f^{n}(x)}f^{m}|_{G_{f^{n}(x)}^{(n)}} \| \leq C \, e^{(\varepsilon n - \mu m)}
      \]
      for all $x \in K$, $n\geq 0$, $m\geq0$. Moreover, if $n-m\geq 0$ then
      \[
       \| ( D_{f^{n-m}(x)}f^{m}|_{H_{f^{n-m}(x)}^{(n-m)}})^{-1} \| \leq C \, e^{\varepsilon n -m(\mu-3\varepsilon)}.
    \]    
    \item [$\rm b_5$)] If $\Pi_x^{(n)}: E \to G_x^{(n)}$ is the projection associated to the splitting $E = H_x^{(n)} \oplus G_x^{(n)}$ then for all $x \in K$ and $n \geq 0$,
    \[
      \|\Pi_{f^{n}(x)}^{(n)} \| \leq C\, e^{\varepsilon n}.
    \]
   \end{itemize}
 \item [c)] There exists $A>0$ such that for all $n\geq0$, \ $d(f^{n}(K), E \setminus U) \geq A\, e^{(\mu+\varepsilon) n}$.
 \end{itemize}
Then there exist $r>0$ and $\delta>0$ arbitrarily small such that if $$\Delta_r= \{(x,v): x\in K, v  \in G_x^{(0)}, \|v\| < r\},$$ there exists a map  $\varphi: \Delta_r \to E$ satisfying the following properties:
\begin{enumerate}
 \item [I)] $\Pi_x^{(0)} \phi(x,v)=v$ for all $(x,v) \in  \Delta_r$.
 \item [II)] For all $(x,v) \in \Delta _r$, $\phi(x,v)$ is the unique vector in $E$ with $\Pi_x^{(0)}\varphi(x,v)=v$ and
 \[
  \|f^{n}(\varphi(x,v)) - f^{n}(x) \| \leq \delta \,e^{-n \mu} 
 \]
for all  $n \geq 0$. In particular, this implies $\varphi(x,0)=x$.
\item [III)] $\varphi$ is continuous, for every $x$ the map $\varphi(x,\cdot)$ is $C^{k+\gamma}$ and its derivatives depend continuously on $(x,v)$. Moreover, its first derivative at $(x,0)$ is the inclusion map.
\end{enumerate}
\end{lemma}

We conclude:
\begin{theorem}
  \label{mainthm-1}
For every $B\in\NN$, the attractor $\hat{\mathcal A}_B$ is uniformly hyperbolic in $\curr$ with two unstable directions. It has a codimension two stable foliation by analytic manifolds. If $\bfv_1$ and $\bfv_2$ belong to the same stable leaf of the foliation then $\cC(\bfv_1)=\cC(\bfv_2)$. In a neighborhood of $\hat{\mathcal A}_B$, the converse is also true.
\end{theorem}
\begin{proof}

  Using Lemma~\ref{ManeLemma} with $K=\hat\cP_B$ we see that if $\mathbf v\in \cP_B$ then there exists an open neighborhood $U(\mathbf v)\subset\curr$ and  $r>0$ such that for every $\mathbf w\in \cP_B\cap U(\mathbf v)$, the stable manifold $W^s_{\text{loc}}(\mathbf w)$ is an analytic graph of codimension two over
  the ball $\Delta$ of radius $r$ around the origin in $T_{\mathbf v}(W^s(\mathbf v))$.
  Since $\hat\cP_B$ is dense in $\hat \cA_B$, passing to converging subsequences, we obtain a codimension two foliation $\cF^s_B$ with analytic leaves through $\hat\cA_B$ which is the closure of the stable foliation of $\hat\cP_B$. Trivial considerations of continuity imply that it has the desired properties.
  \end{proof}

\end{document}